\documentclass[11pt, oneside]{amsart}
\usepackage[title]{appendix}
\usepackage[utf8]{inputenc}
\usepackage{amsmath,amssymb,amsthm,graphics}
\usepackage{colonequals}
\usepackage{bm}
\usepackage{enumitem}
\usepackage{textcomp}
\usepackage[alphabetic]{amsrefs}
\usepackage[all,cmtip]{xy}
\usepackage{yfonts}
\usepackage{stmaryrd}
\usepackage{mathrsfs}
\usepackage{caption}
\usepackage{listings}
\usepackage{comment}
\usepackage[colorlinks,anchorcolor=blue,citecolor=blue,linkcolor=blue,urlcolor =blue,bookmarksdepth=2]{hyperref}
\usepackage{tikz-cd}
\usepackage{mathtools}
\usepackage{ulem, soul}
\usepackage{multirow}
\usepackage{longtable}
\setstcolor{red}

\calclayout

\usepackage{setspace}
\newtheorem{Thm}{Theorem}[section]
\newtheorem{Lem}[Thm]{Lemma}
\newtheorem{Cor}{Corollary}[Thm]
\newtheorem{Prop}[Thm]{Proposition}

\newtheorem{Conj}{Conjecture}[section]
\newtheorem{Question}{Question}[section]

\theoremstyle{definition}
\newtheorem{Def}{Definition}[section]
\newtheorem{eg}[Thm]{Example}

\theoremstyle{remark}
\newtheorem{Rmk}{Remark}[section]

\newenvironment{pf}{\begin{proof}}{\end{proof}}

\numberwithin{equation}{section}

\newcommand{\Z}{\mathbf{Z}}
\newcommand{\Q}{\mathbf{Q}}

\newcommand{\C}{\mathbf{C}}
\newcommand{\F}{\mathbf{F}}

\newcommand{\A}{\text{\bf A}}

\newcommand{\G}{\mathbf{G}}

\newcommand{\NS}{\operatorname{NS}}

\newcommand{\Pic}{\operatorname{Pic}}

\newcommand{\Frob}{\operatorname{Frob}}
\newcommand{\tr}{\operatorname{tr}}
\newcommand{\Sym}{\operatorname{Sym}}
\newcommand{\Spec}{\operatorname{Spec}}
\newcommand{\Tr}{\operatorname{Tran}}

\newcommand{\RNum}[1]{\uppercase\expandafter{\romannumeral #1\relax}}

\DeclareFontFamily{U}{wncy}{}
\DeclareFontShape{U}{wncy}{m}{n}{<->wncyr10}{}
\DeclareSymbolFont{mcy}{U}{wncy}{m}{n}
\DeclareMathSymbol{\Sha}{\mathord}{mcy}{"58} 

\newcommand{\GL}{{\rm GL}}

\newcommand{\Gal}{\mathrm{Gal}}
\newcommand{\GO}{{\rm O}}
\newcommand{\SO}{{\rm SO}}
\newcommand{\SL}{{\rm SL}}

\newcommand{\B}{{\rm B}}

\newcommand{\gr}{{\rm gr}}

\begin{document}

\title[Irreducibility of Galois representations and the Tate conjecture]{Irreducibility of geometric Galois representations and the Tate conjecture for a family of elliptic surfaces}

\author{Lian Duan
	\and Xiyuan Wang}

\address[Lian Duan]{Department of Mathematics, Colorado State University}
\email{lian.duan@colostate.edu}	
	
\address[Xiyuan Wang]{Department of Mathematics, Johns Hopkins University}
\email{xwang151@math.jhu.edu}	

\providecommand{\keywords}[1]{\textbf{\textit{Index terms---}} #1}

\begin{abstract}
 Using Calegari's result on the Fontaine-Mazur conjecture, we study the irreducibility of pure, regular, rank 3 weakly compatible systems of self-dual $\ell$-adic representations. As a consequence, we prove that the Tate conjecture holds for a family of elliptic surfaces defined over $\Q$ with geometric genus bigger than $1$.
\end{abstract}

\maketitle

\tableofcontents
\section{Introduction}

Let $\ell$ be a prime and $K$ be a field. We denote by $G_{K}$ the absolute Galois group of $K$. The study of $\ell$-adic representations of $G_{K}$ is not only interesting in theoretic research in number theory, but also has important application in arithmetic geometry. There are two natural sources of $\ell$-adic Galois representations. The first one arises from the Galois representations attached to algebraic automorphic representations. The second comes from algebraic geometry, i.e., from the subquotient of an $\ell$-adic \'etale cohomology of smooth projective varieties. Both sources of Galois representations are known to be geometric. (For a precise definition of geometric and other basic definitions related to Galois representations appeared in this introduction, we refer the reader to Section \ref{section-background} of this paper.) To study the geometric $\ell$-adic representations, one natural question we can ask is the following.
\begin{Question}\label{Question-irr}
How to tell whether a geometric Galois representation is (absolutely) irreducible?
\end{Question}
It is a conjecture that the geometric Galois representations associated to algebraic cuspidal automorphic representations are irreducible, see \cite{Frank-Gee-irr-of-auto-Gal-repn} for more details. While if a Galois representation is a subquotient of the \'etale cohomology of a smooth projective variety, it is hard to give a satisfactory answer, since this question is closely related to the Grothendieck's theory on pure motives. In this paper, we focus our research on 3-dimensional $\ell$-adic representations of $G_{\Q}$ which come from algebraic geometry, and give a partial answer to this question. As a corollary, for an elliptic surface satisfying the condition $(\ast)$ and $(\ast\ast)$ in Corollary \ref{Cor-Tate-conj-general}, we prove the corresponding Tate conjecture. We also prove that the conditions $(\ast)$ and $(\ast\ast)$ are realizable for a concrete family of elliptic surfaces, and hence prove their Tate conjecture.

\subsection{Main results}

The first main result of this paper provides a representation-theoretic answer to Question~\ref{Question-irr}. Recall that a geometric Galois representation coming from pure motives will induce a weakly compatible system of $\ell$-adic Galois representations (see Definition~\ref{Def-compatible-sys}). 

\begin{Thm}\label{Thm-main-irreducibility} Let $
\{\rho_{\ell}\}_{\ell}$
be a rank $3$ weakly compatible system of self-dual $\ell$-adic representations of $G_{\Q}$ defined over $\Q$. Suppose that it is regular and pure of weight 0. Then either
\begin{enumerate}
    \item $\rho_{\ell}$ is absolutely irreducible for a Dirichlet density one subset of primes $\ell$, or 
    \item for each $\ell$, $\rho_{\ell}$ decomposes into irreducible $\overline{\Q}_{\ell}$-subrepresentations as follows
    \[
    \rho_{\ell}\cong \psi_{\ell}\oplus r_{\ell},
    \]
    where $\psi_{\ell}$ is 1-dimensional and $r_{\ell}$ is 2-dimensional and odd.
\end{enumerate}
\end{Thm}

\begin{Rmk}
This theorem can be deduced from the Fontaine-Mazur conjecture. However, the Fontaine-Mazur conjecture is only known by the work of Kisin \cite{Kisin-fm}, Emerton \cite{Emer-fm}, and Pan \cite{Pan-odd-FM-Conj} for odd $2$-dimensional Galois representations and by Calegari \cite{Even-FM} for even 2-dimensional Galois representations with several additional conditions. (See our discussion about the proof of Theorem \ref{Thm-main-irreducibility}.)
\end{Rmk}

To speak about the application of Theorem~\ref{Thm-main-irreducibility} to arithmetic geometry, let $X$ be a smooth projective variety over $K$. Let $\NS(X_{\overline{K}})$ be the \emph{N\'eron-Severi group} of $X_{\overline{K}}$. There is a $G_{K}$-equivariant cycle class map 
\[
c^{1}: \NS(X_{\overline{K}})\rightarrow H^{2}_{\rm{\acute{e}t}}(X_{\overline{K}}, \Q_{\ell}(1)).
\]
The image of $\NS(X_{\overline{K}})\otimes_{\Z}\Q_{\ell}$ under $c^{1}\otimes \Q_{\ell}$ is called the \emph{algebraic part} of $H^{2}_{\rm{\acute{e}t}}(X_{\overline{K}}, \Q_{\ell}(1))$. It is an $\ell$-adic subrepresentation of $G_{K}$. We define the \emph{transcendental part} $\Tr_{\ell}(X)$ of $H^2_{\rm{\acute{e}t}}(X_{\overline{K}}, \Q_{\ell}(1))$ to be the quotient $H^2_{\rm{\acute{e}t}}(X_{\overline{K}}, \Q_{\ell}(1))/(\NS(X_{\overline{K}})\otimes_{\Z}\Q_{\ell})$. In particular, it is known that when $X$ is an elliptic surface, then $H_{\rm{\acute{e}t}}^{2}(X_{\overline{K}}, \Q_{\ell}(1))\cong \Tr_{\ell}(X)\oplus (\NS(X_{\overline{K}})\otimes_{\Z}\Q_{\ell})$ since the algebraic equivalence classes of such surface is equivalent to its numerical equivalence classes. Since the transcendental part is motivically defined, $\{\Tr_{\ell}(X)\}_{\ell}$ is a weakly compatible system of $\ell$-adic Galois representations. Let $\NS(X)$ be the subgroup of $\NS(X_{\overline{K}})$ generated by the divisors over $K$. We have an induced map
\begin{equation}\label{eqn-cycle-map}
    C^1:\NS(X)\otimes_{\Z}\Q_{\ell} \longrightarrow H^2_{\rm{\acute{e}t}}(X_{\overline{K}}, \Q_{\ell}(1))^{G_K}.
\end{equation}

Tate makes the following conjecture \cite{Tate-Tate-Conj}.
\begin{Conj}[The Tate conjecture for divisors]
Let $K$ be a finitely generated field over its prime field. Then the map $C^{1}$ is an isomorphism. 
\end{Conj}
If $K$ is of characteristic $0$, the above conjecture is known to hold for abelian varieties \cite{Falting-FS}, $K3$ surfaces \cite{Andre-Tate-Conj-K3}, and, more generally, smooth projective varieties with geometric genus $1$ \cite{Moonen-Tate-Conj}. It is also known for elliptic modular surfaces \cite{Shioda-Modular-ES}. If $K$ is of positive characteristic, it holds for abelian varieties \cite{Tate-Tate-Conj-AV} and $K3$ surfaces \cite{Artin-SD-Shafarevich-Tate-Conj}, \cite{Nygaard-Ogus-Tate-Conj-K3} \cite{Maulik-Supersingular-K3}, \cite{Charles-Tate-Conj-finite-fields} and  \cite{Madapusi-Pera-Tate-Conj-K3-odd-char}.

 We call a $G_{K}$-invariant class in $H^{2}_{\rm{\acute{e}t}}(X_{\overline{K}}, \Q_{\ell}(1))$ a Tate class. Roughly speaking, the Tate conjecture claims that the Tate classes are in the algebraic part. If a Tate class is in the transcendental part $\Tr_{\ell}(X)$, it generates an 1-dimensional (trivial) $\ell$-adic subrepresentation. So to prove the Tate conjecture for $X$, it is necessary to show that its transcendental part does not have any 1-dimensional subrepresentation. Based on this idea, we have the following corollary of Theorem \ref{Thm-main-irreducibility}. For an $\ell$-adic representation $\rho$, we will use $\rho^{\rm{ss}}$ to denote its semi-simplification. 
\begin{Cor}\label{Cor-Tate-conj-general}
Let $X\to \mathbb{P}^1_{\Q}$ be a surface defined  over $\Q$ which has an elliptic fibration over $\mathbb{P}_{\Q}^{1}$ and admits a section. Assume that $X$ satisfies the following conditions.
\begin{enumerate}
    \item [$(\ast)$] For some positive integer $s$,
    \[
    \{\Tr_{\ell}(X)^{\rm{ss}}\}_{\ell}\subseteq \bigoplus_{i=1}^{s} \{\rho_{\ell, i}^{\rm{ss}}\}_{\ell},
    \]
    where each $\{\rho_{\ell, i}^{\rm{ss}}\}_{\ell}$ is a regular rank $2$ or $3$ weakly compatible system of self-dual $\ell$-adic representations of $G_{\Q}$ defined over $\Q$.
    \item [$(\ast\ast )$]
    
    If, for any $\ell$ and $i$, $\rho_{\ell, i}^{\rm{ss}}$ is decomposes into irreducible $\overline{\Q}_{\ell}$-subrepresentations as follows
    \[
    \rho_{\ell, i}^{\rm{ss}}\cong \psi_{\ell, i}\oplus r_{\ell, i},
    \]
    with $\dim \psi_{\ell ,i}=1$ and $\dim r_{\ell , i}=2$ , then $\det r_{\ell , i}=1$.
\end{enumerate}

Then, for a Dirichlet density one subset of primes $\ell$, the corresponding Tate conjecture for $X$ is true. Precisely, we have the following isomorphism
\begin{equation}
    \NS(X)\otimes_{\Z}\Q_{\ell}\overset{\sim}{\longrightarrow}H^2_{\rm{\acute{e}t}}(X_{\overline{\Q}}, \Q_{\ell}(1))^{G_{\Q}}.
\end{equation}

\end{Cor}

\begin{Rmk}
Here we explain that $\det r_{\ell, i}=1$ in condition $(\ast \ast)$ is reasonable if $\rho_{\ell, i}$ is in the transcendental part. The self-dual condition of $\rho_{\ell, i}^{\rm{ss}}$ implies that $r_{\ell,i}$ is also self-dual, thus $\det r_{\ell , i}$ is a quadratic character. If this quadratic character is nontrivial, it is a conjecture that there exist a CM elliptic curve $E$ such that $\rho_{\ell , i}^{\rm{ss}}\varepsilon_{\ell}^{-1}\cong \Sym^2 T_{\ell} E$ (up to a quadratic twist). In this case, $\rho_{\ell, i}^{\rm{ss}}$ has a finite image 1-dimensional subrepresentation. This is a contradiction with the Tate conjecture.
\end{Rmk}

In \cite{GT-selfdual}, van Geemen and Top construct a family of non-isotrivial elliptic surfaces $\mathcal{S}_{a}$ parameterized by $a\in \mathbb{P}^1$. Each member in this family has geometric genus 3 and is not an elliptic modular surface.  We apply our method to this family and show the following result.

\begin{Thm}\label{Thm-con-Tate-conj-}
 For each $a\in \Q$, if $a\equiv 2, 3 \mod 5$, and none of $2(1+a)$ or $2(1-a)$ is a square in $\Q$, the surface $\mathcal{S}_{a}$ satisfies the conditions $(\ast)$ and $(\ast\ast)$ in Corollary \ref{Cor-Tate-conj-general}.
 
 In particular, for a Dirichlet density one subset of primes $\ell$, the corresponding Tate conjecture for $\mathcal{S}_{a}$ is true. Precisely, we have the following isomorphism
\begin{equation}\label{eqn-Tate-conj-vGT}
    \NS(\mathcal{S}_{a})\otimes_{\Z}\Q_{\ell}\overset{\sim}{\longrightarrow}H^2_{\rm{\acute{e}t}}((\mathcal{S}_{a})_{\overline{\Q}}, \Q_{\ell}(1))^{G_{\Q}}.
\end{equation}
\end{Thm}

\begin{Rmk}
\begin{enumerate}
    \item To the best of the knowledge of the authors, the conclusion in Theorem~\ref{Thm-con-Tate-conj-} does not follow from the known theory. In fact, the method based on the Kuga-Satake construction requires the varieties to have geometric genus one, and also requires the knowledge of the moduli of such kind of varieties. These prerequisites are not satisfied in our case. In addition, our case is not modular. Although we are not sure if this family can be constructed as quotient of varieties whose Tate conjecture is known, but this is highly nontrivial.
    \item Although our example comes from pullback of $K3$ surfaces, the main contribution of our method is to deal with the transcendental part which is the complement of the transcendental part from the $K3$ surfaces. 
    \item According to the construction of \cite{GT-selfdual}, examples of higher genus can be found. And our method can also be applied to those examples. 
\end{enumerate}
\end{Rmk}

\subsection{Our approach to proving Theorem \ref{Thm-main-irreducibility}}
We first note that $\rho_{\ell}$ is not a direct sum of characters. Otherwise, the regularity of $\rho_{\ell}$ forces this representation to have distinct Hodge-Tate weights, which contracts the pure weight condition. Suppose that $\rho_{\ell}$ has a 2-dimensional absolutely irreducible subrepresentation $r_{\ell}$. If $r_{\ell}$ is odd, by \cite[Theorem~1.0.4]{Pan-odd-FM-Conj}, $r_{\ell}$ lives in a weakly compatible system. So $\{\rho_{\ell}\}_{\ell}$ is a direct sum of 1-dimensional compatible system and an odd 2-dimensional compatible system. If $r_{\ell}$ is even, we apply \cite[Theorem~1.1]{Even-FM} and show that under self-dual condition, all the conditions in Calegari's result will be fulfilled. This implies $r_{\ell}$ is odd, hence contradicts our assumption. And this contradiction completes our proof.

\subsection{Our approach to proving Theorem \ref{Thm-con-Tate-conj-}}

The geometry of $\mathcal{S}_a$ implies that $\Tr_{\ell}(\mathcal{S}_{a})$ generically has a decomposition into three $3$-dimensional subrepresentations. One of them is automatically absolutely irreducible, and the rest two are isomorphic and self-dual when $a\neq \pm 1$. Let $\rho_{\ell}$ be one of the two subrepresentations, and assume $r_{\ell}$ is a $2$-dimensional absolutely irreducible subrepresentation of $\rho_{\ell}$. We show that $r_{\ell}$ is also self-dual and thus by class field theory, there is an integer $D$ such that 
\[
\det r_{\ell}(\Frob_{p}) =\left( \frac{D}{p}\right)
\]
for the prime $p\nmid D$. To prove that $\det r_{\ell}=1$, it is enough to show that $D$ is 1 (up to a square). For the later, we study the relationship between $\tr \rho_{\ell}(\Frob_{p})$ and $\det r_{\ell}(\Frob_p)$ under the self-dual condition. And Theorem~\ref{Thm-con-Tate-conj-} follows immediately after combining our results with counting trick used to compute the trace of $\rho_{\ell}$.

\subsection{Remark on our method}
We want to talk about our method and its potential generalization in motivic aspect. Suppose $r_{\ell}$ exists. By the Tate conjecture, $r_{\ell}$ is not motivically defined. But the Fontaine-Mazur conjecture predicts that $r_{\ell}$ is motivically defined. This is the fundamental contradiction in our proof. To realize the contradiction in our proof, we make use the oddness condition, which reflects the motivic property of geometric Galois representations. Thus in order to generalize our result to non self-dual representations or higher dimensional representations, we expect (1) a proper analog of the oddness condition for higher dimensional representations as well as a geometric method to check this condition;  (2) a generalization of known results which predicts the oddness of geometric Galois representations. Those problems are interesting to the authors. We hope to report a further result in this direction in a future paper.

\subsection{Outline of the paper}In Section \ref{section-background}, we collect necessary definitions and facts on Galois representations. In particular, a theorem of Calegari about the Fontaine-Mazur conjecture is mentioned. In Section~\ref{section-pf-Thm-irr-dim3} we prove Theorem~\ref{Thm-main-irreducibility}. Then, as an application of this theorem, we use it to prove the Tate conjecture for the elliptic surfaces satisfying the conditions $(\ast)$ and $(\ast\ast)$. In Section~\ref{section-Tate-conj-vGT}, we first recall a family of elliptic surfaces constructed in \cite{GT-selfdual}. Then we verify that the conditions $(\ast)$ and $(\ast\ast)$ of Corollary~\ref{Cor-Tate-conj-general} are satisfied for about $40\%$ members in this family. Then as a corollary we prove the Tate conjecture for those members.

\subsection{Notations and conventions}\label{Section-notation}
For a field $K$, we fix the separable closure $\overline{K}$ of $K$. If $K$ is a number field and $\mathfrak{p}$ is a finite place of $K$, we let $\Frob_{\mathfrak{p}}$ denote the geometric Frobenius.
    
If $X$ is a $K$-scheme, we let $X_{\overline{K}}$ denote the base-change $X\times_{\Spec K}\Spec \overline{K}$. The symbol $\dim$ in this paper means the dimension over $\overline{\Q}_{\ell}$

For a rational number $a=\frac{n}{m}$ with $\gcd(m,n)=1$, we say $p$ is a divisor of $a$ if either $p|m$ or $p|n$. And we denote by $\left(\frac{a}{p}\right)$ the classical Legendre symbol $\left(\frac{mn}{p}\right)$.

\section{Backgrounds of Galois representations}\label{section-background}
In this section, we recall some definitions and facts on Galois representations. The readers who are familiar with Galois representations can skip this section.

Let $L$ be a topological field and $V$ be a finite dimensional topological vector space over $L$. A \emph{Galois representation} (or a \emph{$L$-representation} of $G_{K}$) is a continuous linear group action of $G_{K}$ on $V$. Up to a choice of basis of $V$, we can realize this representation as a continuous homomorphism
$$
\rho: G_K\longrightarrow \GL_n(L).
$$
Such data is denoted by $\{V, \rho\}$ (or one of $V$ or $\rho$ for simple). If $L$ is a finite extension of $\Q_{\ell}$ or $\overline{\Q}_{\ell}$, we call $\rho$ an \emph{$\ell$-adic representation}.

\subsection{\texorpdfstring{$\ell$}{}-adic representations}
We begin with recalling some basic definitions in $\ell$-adic Hodge theory. In this paper, we will only need those definitions formally. We refer the reader to \cite{Fontaine-Ouyang-p-adic-Gal-repn} for details. Suppose $K/\Q_{\ell}$ is a finite extension. Let $\rm{B}_{\rm{dR}}$ be Fontaine's de Rham periods ring. It is a filtered $K$-algebra with a continuous $K$-linear action of $G_{K}$.

\begin{Def}
Let $V$ be a $\overline{\Q}_{\ell}$-representation of $G_{K}$. We say $V$ is \emph{de Rham} if
$\dim_{\overline{\Q}_{\ell}}(\B_{\rm{dR}}\otimes_{K, \tau}V)^{G_{K}}=\dim_{\overline{\Q}_{\ell}}V$ for all $\Q_{\ell}$-embedding $\tau: K\hookrightarrow \overline{\Q}_{\ell}$. If $V$ is de Rham, for each $\Q_{\ell}$-embedding $\tau:K\hookrightarrow \overline{\Q}_{\ell}$, we define the $\dim_{\overline{\Q}_{\ell}}V$-element multi-set of \emph{$\tau$-Hodge-Tate weights}, $\rm{HT}_{\tau} (V)$, to be the multi-set of integers $h$ such that  
\[
\gr^{h}(\rm{B}_{dR}\otimes_{K, \tau}V)^{G_{K}} \not=0
\]
where $h$ has multiplicity $\dim_{\overline{\Q}_{\ell}}\gr^{h}(\rm{B}_{dR}\otimes_{K, \tau}V)^{G_{K}}$.
\end{Def}

\begin{eg}
Let $\varepsilon_{\ell}: G_{\Q}\rightarrow \overline{\Q}^{\times}_{\ell}$ be the $\ell$-adic cyclotomic character. Then $\varepsilon_{\ell}|_{G_{\Q_{\ell}}}$ is de Rham. The Hodge-Tate weight of $\varepsilon_{\ell}$ (more precisely, $\varepsilon_{\ell}|_{G_{\Q_{\ell}}}$) is $-1$.
\end{eg}

Now we can talk about global $\ell$-adic representations.
\begin{Def}\label{Def-geo-Gal-repn}
Let $K/\Q$ be a finite extension and $\rho$ be an $\ell$-adic representation of $G_{K}$. 
We say $\rho$ is \emph{geometric} if $\rho$ is unramified almost everywhere and $\rho|_{G_{K_{v}}}$ is de Rham for every place $v$ of $K$ above $\ell$. 
\end{Def}

\begin{eg}
Suppose that $X$ is a smooth projective variety over a number field $K$. Then the $\ell$-adic representation $H_{\rm{\acute{e}t}}^{i}(X_{\overline{K}}, \Q_{\ell})$ of $G_{K}$, for $0\leq i\leq 2\dim X$, is geometric. Furthermore, any subquotient of the $\ell$-adic representation $H_{\rm{\acute{e}t}}^{i}(X_{\overline{K}}, \Q_{\ell})$ is geometric.
\end{eg}

If $V$ is an $\ell$-adic representation, we use  $V(n)$ to denote $V$ tensored with the $n$th power of the $\ell$-adic cyclotomic character.

\begin{Lem}\label{Cor-dim1-repn}
Let $K/\Q$ is a totally real field. Suppose $\rho: G_{K} \to \overline{\Q}_{\ell}^{\times}$ is a geometric $\ell$-adic representation.
Then $$\rho =\tau \cdot \varepsilon_{\ell}^{n}|_{G_{K}} $$ for some non-negative integer $n$, where $\varepsilon_{\ell}$ is the $\ell$-adic cyclotomic character of $G_{\Q}$ and $\tau $ is of finite order.
\end{Lem}

\begin{pf}
By class field theory, it is enough to study the algebraic Hecke characters of $\A_{K}^{\times}/K^{\times}$. Then this lemma follows from a classification of such characters \cite{Weil-certian-type-char-idele}.
\end{pf}
Let $c\in G_{\Q}$ be a fixed complex conjugation. For an $\ell$-adic representation of $G_{\Q}$, $\det\rho(c)\in \{1,-1\}$.  
\begin{Def}\label{Def-odd-even-repn}
We say an $\ell$-adic representation $\rho$ of $G_{\Q}$ is \emph{odd} (resp. \emph{even}) if $\det\rho(c)$ is $-1$ (resp. $1$).
\end{Def}

\begin{Def}\label{Def-selfdual-repn}
We say a Galois representation $\rho$ is \emph{self-dual} if 
$
\rho\simeq \rho^*
$
, where $\rho^{*}$ is the dual of $\rho$.
\end{Def}
One of the main problem concerning the geometric $\ell$-adic representations is the Fontaine-Mazur conjecture. In this paper, an important input is Calegari's result on Fontaine-Mazur conjecture (see \cite[Theorem~1.1]{Even-FM}), which we state here for the convenience of readers.
\begin{Thm}\label{Thm-Frank-FM-conj}
Let $r: G_{\Q}\rightarrow \GL_{2}(\overline{\Q}_{\ell})$ be an $\ell$-adic representation. Suppose that $\ell>7$, and, furthermore, that 
\begin{enumerate}
    \item[(a)] $r$ is geometric, i.e., unramified almost everywhere and de Rham at $\ell$.
    \item[(b)] $r|_{G_{\Q_{\ell}}}$ has distinct Hodge-Tate weights.
    \item[(c)] $\overline{r}|_{G_{\Q_{\ell}}}$ is not a twist of a representation of the form 
    $$\begin{pmatrix}
    \overline{\varepsilon}_{\ell} & \ast\\
    0 & 1
    \end{pmatrix}$$
    where $\overline{\varepsilon}_{\ell}$ is the mod $\ell$ cyclotomic character.
     \item[(d)] The residue representation $\overline{r}$ is not of dihedral type.
      \item[(e)] The residue representation $\overline{r}$ is absolutely irreducible.
\end{enumerate}
Then $r$ is modular. In particular, $r$ is odd. 
\end{Thm}

\subsection{Weakly compatible system of  \texorpdfstring{$\ell$}{}-adic representations}

The following definition of compatible  system follows from \cite[Section~5.1]{BLGGT-2014}. For the convenience of readers, we also state it here. 
\begin{Def}\label{Def-compatible-sys}
    Let $K$ denote a number field. A \emph{rank $n$ weakly compatible system of $\ell$-adic representations $\mathcal{R}$ of $G_K$ defined over $M$} is a  $5$-tuple 
    $$
    (M, S, \{Q_{v}(T)\}, \{\rho_{\lambda}\}, \{H_{\tau}\}),
    $$
    where 
    \begin{enumerate}
        \item $M$ is a number field.
        \item $S$ is a finite set of primes of $K$. 
        \item for each $v\notin S$, $Q_{v}(T)$ is a monic degree $n$ polynomial in $M[T]$.
        \item for each prime $\lambda$ of $M$ (with residue characteristic $\ell$)
        $$
        \rho_{\lambda}: G_K\to \GL_n(\overline{M}_{\lambda})
        $$
        is a continuous, semi-simple representation such that
        \begin{itemize}
            \item if $v\notin S$ and $v\nmid \ell$, is a prime of $K$, then $\rho_{\lambda}$ is unramified at $v$ and $\rho_{\lambda}(\Frob_{v})$ has characteristic polynomial $Q_{v}(T)$.
            \item if $v|\ell$ then $\rho_{\lambda}|_{G_{K_v}}$ is de Rham.
        \end{itemize}
        \item for $\tau: K\hookrightarrow \overline{M}$, $H_{\tau}$ is a multiset of $n$ integers such that  for any $\overline{M}\hookrightarrow \overline{M}_{\lambda}$ over $M$ we have $\rm HT_{\tau}(\rho_{\lambda})=H_{\tau}$.
    \end{enumerate}

    We will call $\mathcal{R}$ \emph{regular} if for each $\tau: K\hookrightarrow \overline{M}$ every element of $H_{\tau}$ has multiplicity 1. Suppose that $K$ is $\Q$ and $n$ is 2, we will call $\mathcal{R}$ is \emph{odd} if every representation in $\mathcal{R}$ is odd.  
    
    We will call $\mathcal{R}$ \emph{pure of weight $w$} if 
    \begin{enumerate}
        \item [$\bullet$] for each $v\not\in S$, each root $\alpha$ of $Q_{v}(T)$ in $\overline{M}$ and each $\iota : \overline{M}\hookrightarrow \C$ we have 
        \[
        |\iota \alpha|^2=q_{v}^{w},
        \]
        where $q_{v}$ is the cardinality of the residue field of $K_{v}$;
        \item [$\bullet$] and for each $\tau: K\hookrightarrow \overline{M}$ and each complex conjugation $c$ in $Gal(\overline{M}/ \Q)$ we have 
        \[
        H_{c\tau}=\{w-h: h\in H_{\tau}\}.
        \]
    \end{enumerate}
    
\end{Def}
We will sometimes simply write $\{\rho_{\lambda}\}_{\lambda}$ for a weakly compatible system $\mathcal{R}$.

\begin{eg}
Let $X$ be a projective smooth variety over a number field $K$, then for $0\leq i\leq 2\dim X$, $\{H^i_{\rm{\acute{e}t}}(X_{\overline{X}}, \Q_{\ell})\}_{\ell}$ is a weakly compatible system of $\ell$-adic representations of $G_{K}$ defined over $\Q$.
\end{eg}

Let $\mathcal{R}_{1}=(M, S_1, \{Q_{v, 1}(T)\}, \{\rho_{v, 1}\}, \{H_{\tau, 1}\})$ and $\mathcal{R}_{2}=(M, S_2, \{Q_{v,2}(T)\}, \{\rho_{v, 2}\}, \{H_{\tau, 2}\})$ be two weakly compatible systems of $\ell$-adic representations of $G_{K}$ defined over $M$. We can define direct sum $\mathcal{R}_{1}\oplus \mathcal{R}_{2}$ a new weakly compatible system of $\ell$-adic representation of $G_{K}$ defined over $M$ by 
\[
(M, S_1\cup S_2, \{Q_{v,1}(T)Q_{v,2}(T)\}, \{\rho_{v,1}\oplus \rho_{v,2}\}, \{H_{\tau ,1}\sqcup H_{\tau, 2} \} \}).
\]

\section{Irreducibility of \texorpdfstring{$\ell$}{}-adic representations}\label{section-pf-Thm-irr-dim3}

In this section, we will prove Theorem \ref{Thm-main-irreducibility} and Corollary~\ref{Cor-Tate-conj-general}. We will fix $\{\rho_{\ell}\}_{\ell}$ to be the weakly compatible system in Theorem~\ref{Thm-main-irreducibility} and take $T$ to be the exceptional set which consists of primes $\ell$ where $\rho_{\ell}$ is not absolutely irreducible. Assuming that the Dirichlet density of $T$ is greater than 0, we will show that there exists an $\ell\in T$ such that $\rho_{\ell}: G_{\Q}\rightarrow \GL_{3}(\overline{\Q}_{\ell})$ has an odd 2-dimensional irreducible subrepresentation and thus $T$ is the whole set of rational primes.
To state our strategy, notice that since $\rho_{\ell}$ is not absolutely irreducible for $\ell\in T$, $\rho_{\ell}$ has one of the two decompositions 
$$
 \chi_{\ell,1}\oplus \chi_{\ell,2}\oplus \chi_{\ell,3} ,\quad \text{ or } \quad \psi_{\ell}\oplus r_{\ell},
$$
where $\chi_{\ell,i}$ ($i=1,2,3$) and $\psi_{\ell}$ are $1$-dimensional $\overline{\Q}_{\ell}$-representations and $r_{\ell}$ is an irreducible $2$-dimensional $\overline{\Q}_{\ell}$-representation. In the following, we will say the former case is \emph{of type $1+1+1$}, and say later case is \emph{of type $1+2$}. This definition can be generalized to a $3$-dimensional $G_K$-representation for a number field $K$. In subsection \ref{Section-pf-main-irr} we first exclude the possibility for $\rho_{\ell}$ to be of type $1+1+1$. Then we show that, if $\rho_{\ell}$ is of type $1+2$ and the 2-dimensional subrepresentation $r_{\ell}$ is odd, $\{\rho_{\ell}\}_{\ell}$ is in fact a direct sum of a rank 1 weakly compatible system and a rank 2 odd weakly compatible system. Finally we show the nonexistence of the `` even $2$-dimensional subrepresentation $r_{\ell}$'' up to a density zero subset of all primes. So Theorem \ref{Thm-main-irreducibility} follows. In subsection \ref{section-Tate-conj-general},  as an application of Theorem~\ref{Thm-main-irreducibility}, we prove Corollary~\ref{Cor-Tate-conj-general}.

\subsection{Proof of Theorem \ref{Thm-main-irreducibility}}\label{Section-pf-main-irr}
In this subsection we will specialize $\{\rho_{\ell}\}_{\ell}$ to be the regular weakly compatible system of self-dual representations in Theorem~\ref{Thm-main-irreducibility}, and let $T$ be the set of primes where $\rho_{\ell}$ is of type $1+2$. In order to prove Theorem~\ref{Thm-main-irreducibility}, we first assume that $\rho_{\ell}:G_{\Q}\rightarrow \GL_3(\overline{\Q}_{\ell})$ is of type $1+1+1$ and try to deduce a contradiction. Note that if this was the case, then for every Galois extension $K/\Q$, the restriction $\rho_{\ell}|_{G_{K}}$ is also of type $1+1+1$.  

\begin{Prop}\label{Prop-1+1+1-type}
Let $K/\Q$ be a totally real extension. Let $\{\pi_{\ell}\}_{\ell}$ be a rank $3$ weakly compatible system of $G_K$ defined over $\Q$. If $\{\pi_{\ell}\}_{\ell}$ is regular and pure of weight $0$, then $\pi_{\ell}$ is not of type $1+1+1$ for any $\ell$. In particular, under the condition of Theorem \ref{Thm-main-irreducibility}, $\rho_{\ell}|_{G_{K}}$ is not of type $1+1+1$.
\end{Prop}


\begin{pf}[proof of Proposition \ref{Prop-1+1+1-type}]
    Suppose $\pi_{\ell_0}\cong \pi_{\ell_0, 1}\oplus \pi_{\ell_0, 2}\oplus \pi_{\ell_0, 3}$ is of type $1+1+1$ for a fixed prime $\ell_0$. Then by Lemma~\ref{Cor-dim1-repn} we know that for each $i\in \{1,2,3\}$, $\pi_{\ell_0, i}\cong \chi_i\varepsilon^{n_i}$ has to be a product of a finite character $\chi_i$ and a power of cyclotomic character $\varepsilon$. On one side, by the pure of weight zero condition, all the $n_i$ are the same. However, by the regularity, $n_i$'s are distinct. This contradiction implies that $\pi_{\ell_0}$ cannot be of type $1+1+1$.
\end{pf}

Now we start to discuss the type $1+2$ case.  First we prove a lemma about the oddness condition in our setting.
\begin{Lem}\label{det}
If $\rho_{\ell}$ is of type $1+2$, then the 2-dimensional subrepresentation $r_{\ell}$ is odd if and only if $\det r_{\ell}$ is nontrivial.   
\end{Lem}
\begin{proof}
The self-dual condition of $\rho_{\ell}$ implies that $r_{\ell}$ is also self-dual, thus $\det r_{\ell}$ is a quadratic character. Since it is a fact that the image of a self-dual $2$-dimensional representation is contained in either $\GO_2(\overline{\Q}_{\ell})$ or in $\SL_2(\overline{\Q}_{\ell})$, suppose that $\det r_{\ell}$ is nontrivial, then $r_{\ell}(G_{\Q})$ is not in $ \SL_{2}(\overline{\Q}_{\ell})$. Let $K$ be the quadratic field fixed by the kernel of $\det r_{\ell}$, then we have $r_{\ell}(G_{K})\subset \SO_{2}(\overline{\Q}_{\ell})$ is an abelian group. Then $r_{\ell}|_{G_{K}}$ is not irreducible. So $\rho_{\ell}|_{G_{K}}$ is of type $1+1+1$. By Proposition \ref{Prop-1+1+1-type}, $K$ is imaginary, hence $r_{\ell}$ is odd. The other direction is easy.
\end{proof}

\begin{Prop}\label{Odd-compatible}
If $\rho_{\ell_0}$ is of type $1+2$ with an odd $2$-dimensional subrepresentation $r_{\ell_0}$ for some prime $\ell_0$, then so is $\rho_{\ell}$ for every prime $\ell$.
\end{Prop}
\begin{proof}
By assumption, we have  $\rho_{\ell_0}\cong \psi_{\ell_0}\oplus r_{\ell_0}$ where $r_{\ell_0}$ is of $2$-dimension and is absolutely irreducible and odd. By \cite[Theorem~1.0.4]{Pan-odd-FM-Conj}, $r_{\ell_0}$ comes from a cuspidal eigenform. So $r_{\ell_0}$ lives in an odd weakly compatible system $\{r_{\ell}\}_{\ell}$. We can also extend $\psi_{\ell_0}$ to a weakly compatible system since $\rho_{\ell}$ is geometric. So $\{\rho_{\ell}\}_{\ell}\cong \{\psi_{\ell}\}_{\ell}\oplus \{r_{\ell}\}_{\ell}$.
\end{proof}

The above Proposition~\ref{Odd-compatible} proves possibility $(2)$ in Theorem~\ref{Thm-main-irreducibility}. And combining this proposition with the Proposition~\ref{Prop-1+2-type} below, we can complete the proof of Theorem~\ref{Thm-main-irreducibility}. Recall that the exceptional set is defined by $T=\{\ell \text{ prime}| \rho_{\ell} \text{ is of type }1+2\}$.

\begin{Prop}\label{Prop-1+2-type}
    If $T$ has positive Dirichlet density, then there exists at least one $\ell\in T$ such that $r_{\ell}$ is odd.
\end{Prop}

The strategy of proving Proposition \ref{Prop-1+2-type} is proof by contradiction. Assume that Proposition \ref{Prop-1+2-type} is false, then $T$ has positive density and for every $\ell\in T$, $\rho_{\ell}$ has a 2-dimensional irreducible even subrepresentation $r_{\ell}:G_{\Q}\rightarrow \GL_{2}(\overline{\Q}_{\ell})$. By Lemma~\ref{det}, $\det r_{\ell}=1$ for $\ell\in T$. (This is really the condition we will use later.) In the rest part of this subsection, we will use Calegari's theorem on the Fontaine-Mazur conjecture (Theorem \ref{Thm-Frank-FM-conj}) to show that for some $\ell\in T$ the corresponding $r_{\ell}$ are also odd, which is impossible, and hence deduce to contradiction.

In order to use Theorem \ref{Thm-Frank-FM-conj}, without loss of generality, we can assume all primes in $T$ are greater than $7$, then, under the conditions of Theorem \ref{Thm-main-irreducibility}, we need to show that there is a subset $T'\subset T$ consisting of infinity many $\ell$ such that $r_{\ell}$ satisfies all the five conditions of Theorem \ref{Thm-Frank-FM-conj}.

\begin{Lem}\label{Lem-conditions-1-2}
 For $\ell\in T$, $r_{\ell}$ is geometric and has distinct Hodge-Tate weights. Equivalently, conditions $(a)$ and $(b)$ of Theorem \ref{Thm-Frank-FM-conj} are true for $r_{\ell}$ with $\ell\in T$.
\end{Lem}

\begin{pf}
Every $\rho_{\ell}$ is geometric, and so is $r_{\ell}$ since the geometric property is closed under taking subquotient, this proves the condition $(a)$. 

Recall that the weakly compatible system $\{\rho_{\ell}\}_{\ell}$ is regular. So $\rho_{\ell}$ has distinct Hodge-Tate weights. The $r_{\ell}$ is a subquotient of $\rho_{\ell}$, hence condition $(b)$ is true. Indeed, one easily sees that the Hodge-Tate weights of $r_{\ell}$ are $\{m,-m\}$ for some nonzero integer $m$.
\end{pf}
 
\begin{Lem}\label{Lem-condition-3}
Assume that Proposition \ref{Prop-1+2-type} is false. The representation $\overline{r}_{\ell}|_{G_{\Q_{\ell}}}$ is not a twist of an extension of trivial character by mod $\ell$ cyclotomic character $\overline{\varepsilon}_{\ell}$. In particular, condition $(c)$ of Theorem \ref{Thm-Frank-FM-conj} is true for $r_{\ell}$ with $\ell\in T$. 
\end{Lem}    
\begin{proof}
Note that the mod $\ell$ cyclotomic character $\overline{\varepsilon}_{\ell}$ is a surjective map to $\F_{\ell}^{\times}$. Taking $g\in G_{\Q_{\ell}}$ such that $\overline{\varepsilon}_{\ell}(g)\in \F_{\ell}^{\times}$ is not a square. If $\overline{r}_{\ell}|_{G_{\Q_{\ell}}}$ is a twist of an extension of trivial character by mod $\ell$ cyclotomic character, then $\det\overline{r}_{\ell}|_{G_{\Q_{\ell}}}(g)$ is not a square element. So $\det r_{\ell}$ is nontrivial. This is a contradiction.
\end{proof}

\begin{Lem}\label{Lem-condition-4}
There exists a subset $T'$ of Dirichlet density one with respect to $T$, such that $\overline{r}_{\ell}$ is not of dihedral type for any $\ell\in T'$. Hence condition $(d)$ of Theorem~\ref{Thm-Frank-FM-conj} is true for $r_{\ell}$ with $\ell\in T'$.
\end{Lem}
\begin{proof}
By Proposition 2.7 in \cite{Frank-Gee-irr-of-auto-Gal-repn}, there is a Dirichlet density one subset $J$ of primes, such that if $\rho_{\lambda}$ has a 2-dimensional even absolutely irreducible subrepresentation $r_{\lambda}$ for $\lambda\in J$, then $\overline{r}_{\lambda}$ is not of dihedral type. Then the set $T^{'}=J\cap T$ is our desired subset.

\end{proof}
At last, we need to speak about the irreducibility of $\overline{r}_{\ell}$. 

 \begin{Lem}\label{Lem-condition-5}
 There exists a subset $T''$ of Dirichlet density one with respect to T, such that the residue representation $\overline{r}_{\ell}$ is absolutely irreducible for any $\ell\in T^{''}$. Equivalently, condition $(e)$ of Theorem~\ref{Thm-Frank-FM-conj} is true for $r_{\ell}$ with $\ell\in T^{''}$.
    \end{Lem}

\begin{proof}
This is a corollary of Proposition 5.3.2 of \cite{BLGGT-2014}.
\end{proof}

    Finally, by combing all the above discussions, we can complete the proofs to Proposition~\ref{Prop-1+2-type} and  Theorem~\ref{Thm-main-irreducibility}.
    
    \begin{pf}[Proof of Proposition \ref{Prop-1+2-type}]
       
    Assume that Proposition \ref{Prop-1+2-type} is false, then $T$ has positive density, and for every $\ell\in T$, $\rho_{\ell}$ has a 2-dimensional irreducible even subrepresentation $r_{\ell}:G_{\Q}\rightarrow \GL_{2}(\overline{\Q}_{\ell})$. Following the Lemmas~\ref{Lem-conditions-1-2}, \ref{Lem-condition-3}, \ref{Lem-condition-4} and \ref{Lem-condition-5} and the notations in their proofs, we can find a density one subset $\widetilde{T}:=T'\cap T''$ of $T$ such that $r_{\ell}$ satisfies all assumptions of Theorem \ref{Thm-Frank-FM-conj} for every $\ell\in \widetilde{T}$. Hence, for $\ell\in \widetilde{T}$, $r_{\ell}$ is modular, hence odd. However, this is impossible since by our setups. Thus by all above, we cannot assume $r_{\ell}$ is even for every $\ell\in T$, and this completes the proof of Proposition~\ref{Prop-1+2-type}
    \end{pf}
    
    \begin{pf}[Proof of Theorem~\ref{Thm-main-irreducibility}]
       This theorem follows by combining Proposition~\ref{Prop-1+1+1-type}, Proposition \ref{Odd-compatible}, and Proposition~\ref{Prop-1+2-type}.
    \end{pf}

\subsection{Proof of Corollary~\ref{Cor-Tate-conj-general}}\label{section-Tate-conj-general}
Following the notations of Corollary~\ref{Cor-Tate-conj-general}, by Theorem \ref{Thm-main-irreducibility} and condition $(\ast\ast)$ of Corollary \ref{Cor-Tate-conj-general}, we know that for a Dirichlet density one set of primes $\ell$, all 3-dimensional $G_{\Q}$-representations $\rho_{\ell, i}^{\rm{ss}}$ are absolutely irreducible. One can also easily argue that all the 2-dimensional $G_{\Q}$-representations $\rho_{\ell, i}^{\rm{ss}}$ are absolutely irreducible following exactly the same idea in the proof of Proposition \ref{Prop-1+1+1-type}. This means these transcendental parts have no contribution to the Galois invariant part of $H^2_{\rm{\acute{e}t}}(X_{\overline{\Q}}, \Q_{\ell}(1))$. Thus to prove Corollary~\ref{Cor-Tate-conj-general}, it is sufficient to show that 
\begin{equation}\label{eqn-cycle-map-NS}
    \NS(X)\otimes \Q_{\ell}\rightarrow (\NS(X_{\overline{\Q}})\otimes\Q_{\ell})^{G_{\Q}}
\end{equation}
is an isomorphism. Indeed this is true when $X$ is an elliptic surface over base curve $\mathbb{P}^1$. This result should be known for experts. But for the convenience of readers, we write a proof in this section. Our proof is not original, it follows the idea in \cite[Chapter~17, Section~3]{K3-lecture-notes}, especially its Remark~3.2.

\begin{Prop}\label{Prop-NS-descend}
    If $X\to \mathbb{P}^1_{\Q}$ is a surface defined over $\Q$ which has an elliptic fibration over $\mathbb{P}_{\Q}^1$ and admits a section, then $(\ref{eqn-cycle-map-NS})$ holds.
\end{Prop}

\begin{pf}
   Note that $\NS(X_{\overline{\Q}})$ is generated by the geometric divisors of $X_{\overline{\Q}}$.  This means that we can find a finite Galois extension $K/\Q$ such that all the generators are defined over $K$, i.e. 
    $$
    \NS(X_{\overline{\Q}})=\NS(X_{K}).
    $$
    Thus to prove the corollary, it is reduced to show 
    $$
    (\NS(X_K)\otimes \Q_{\ell})^{G_{\Q}}=\NS(X)\otimes \Q_{\ell}.
    $$
    To show this, we note that for elliptic surfaces with section over a base curve of genus $0$, we have $\NS(X_{\overline{\Q}})=\Pic(X_{\overline{\Q}})$, i.e. the linearly equivalent class and the algebraic equivalent class coincide by the fact that the pull back map from $\Pic^0(C_{\overline{\Q}})$ to $\Pic^0(X_{\overline{\Q}})$ is an isomorphism (here $C$ is the genus $0$ base curve of $X$). Note that $\Pic(X)=H^1(X, \G_m)$, and similarly for $\Pic(X_K)$ and consider the Hochschild-Serre spectral sequence 
    $$
    E_2^{p,q}=H^p(\Gal(K/\Q), H^q(X_K, \G_{m}))\Rightarrow H^{p+q}(X, \G_m)
    $$
    and apply the facts that $H^1(\Gal(K/\Q), K^*)=0$ (i.e. Hilbert 90) we get 
    $$
    0\to \Pic(X)\to \Pic(X_K)^{\Gal(K/\Q)}\to H^2(\Gal(K/\Q), K^*).
    $$
    According to the fact that $H^2(\Gal(K/\Q), K^*)$ is torsion, we have 
    $$
    (\Pic(X_K)\otimes \Q_{\ell})^{G_{\Q}}=\Pic(X)\otimes \Q_{\ell}.
    $$
    This finishes the proof.
\end{pf}

\section{The Tate conjecture of the surfaces of van Geemen and Top}\label{section-Tate-conj-vGT}

In this section, we will apply Corollary~\ref{Cor-Tate-conj-general} to the construction of van Geemen and Top \cite{GT-selfdual}. As a result, we show that the Tate conjecture is true for a sub family of their construction. 

More precisely, in Section \ref{section-surface-vGT} we recall the construction of a family of elliptic surfaces $\mathcal{S}_{a}$ in \cite{GT-selfdual}, and the properties of the corresponding weakly compatible system $\{ V_{\ell}(1)\}_{\ell}$ which is constructed by van Geemen and Top. (Recall that $V_{\ell}(1)\cong V_{\ell}\otimes \varepsilon_{\ell}$.) In Section \ref{section-Trace-Vl}, we apply some calculation tricks to work out several technical properties of the trace of $V_{\ell}$ which can be used in proving Theorem~\ref{Thm-con-Tate-conj-}. Finally, in Section \ref{section-pf-irr-vGT} we prove Theorem~\ref{Thm-con-Tate-conj-}.

\subsection{The surfaces of van Geemen and Top}\label{section-surface-vGT} 
We simply recall the construction of van Geemen and Top here. Readers who are interested in more details are referred to Section 2 and Section 5 of the original paper.

For each $a\in \Q\setminus\{ \pm 1\}$, considering the elliptic surface
    \begin{equation}\label{eqn-def-eqn-E}
         \mathcal{E}_{a}: Y^2=X\left(X^2+2\left(\frac{a+1}{t^2}+a\right)X+1\right).
    \end{equation}
    \begin{Rmk}
    The original surface in their paper has two parameters $a$ and $s$, while when $s \in \Q^{*}$, we can parameterize the equation to get the form as above.
    \end{Rmk}
    
    Define the elliptic surfaces $\mathcal{X}_{a}$ and $\mathcal{S}_{a}$ as fiber products of $\mathcal{E}_{a}$ which satisfies the following Cartesian diagram

\[
\begin{tikzcd}
\mathcal{S}_{a} \arrow[r] \arrow[d]
& \mathcal{X}_{a} \arrow[r]  \arrow[d]  & \mathcal{E}_{a}  \arrow[d]\\
\mathbb{P}^1_{z} \arrow[r, "j" ]
& \mathbb{P}^1_{u} \arrow[r, "h"] & \mathbb{P}^1_{t}
\end{tikzcd}
\] 
where $j: z\mapsto u=(z^2-1)/z$ and $h: u\mapsto t=(u^2-4)/(4u)$. One can see that $\mathcal{S}_a$ is not isotrivial by computing its $j$-invariant. In this and  the following sections, we will denote by $\mathcal{E}_{a,t}$ (resp. $\mathcal{X}_{a,u}$, $\mathcal{S}_{a,z}$) the fiber above $t\in \mathbb{P}^1(\overline{\Q})$ (resp. $u$, $z$) of the surface $\mathcal{E}_a$ (resp. $\mathcal{X}_a$, $\mathcal{S}_a$). 

Considering the geometric action on $\mathbb{P}^1_z$ defined by 
$$
\sigma: z\mapsto \frac{z+1}{-z+1}
$$
which has order $4$. As an element of the Galois group of $\mathbb{P}^1_{z}$ over $\mathbb{P}^1_{t}$, $\sigma$ is totally ramified over $t=\pm i$ and this is all the ramifications. Also, one sees that $j:\mathbb{P}_z^1\to  \mathbb{P}_u^1$ defined above identifies $\mathbb{P}_u^1$ with the quotient space $\mathbb{P}_z^1/\langle \sigma^2 \rangle$. This means that both of the two \'etale cohomology $H^2_{\rm{\acute{e}t}}(({\mathcal{S}_a})_{\overline{\Q}}, \Q_{\ell})$ and $H^2_{\rm{\acute{e}t}}((\mathcal{X}_{a})_{\overline{\Q}}, \Q_{\ell})$ are stable under the $\G_{\Q}\times \langle\sigma\rangle$-action. Moreover, if we denote by $A_{\ell}({\mathcal{S}_a})$ the $\Q_{\ell}$-subspace in $H^2_{\rm{\acute{e}t}}((\mathcal{S}_a)_{\overline{\Q}}, \Q_{\ell})$ spanned by all components of bad fibers of ${\mathcal{S}_a}\to \mathbb{P}^1_{\overline{\Q}}$, then $A_{\ell}({\mathcal{S}_a})$ is also $\G_{\Q}\times \langle\sigma\rangle$-stable. 

Define 
$$
W_{\ell}:=H^2_{\rm{\acute{e}t}}((\mathcal{S}_a)_{\overline{\Q}}, \Q_{\ell})/(H^2_{\rm{\acute{e}t}}((\mathcal{X}_a)_{\overline{\Q}}, \Q_{\ell})+ A_{\ell}({\mathcal{S}_a})).
$$
Then $W_{\ell}$ has dimension $6$ and is also equipped with a $\G_{\Q}\times \langle\sigma\rangle$-action. The $\sigma$ has two eigenvalues $\pm i$ on $W_{\ell}$. We take $V_{\ell}$ (resp. $\overline{V}_{\ell}$) to be the 3-dimensional eigenspace corresponding to eigenvalue $i$ (resp. $-i$).
\begin{Rmk}
From now on, we work on the representation $V_{\ell}$. But all the following results are also true for $\overline{V}_{\ell}$ due to the isomorphism between $\overline{V}_{\ell}$ and $V_{\ell}$.
\end{Rmk}

\begin{Prop}\cite[Proposition~5.2]{GT-selfdual}
    For $a\in \Q\setminus \{\pm1\}$, the corresponding $V_{\ell}(1)$ is self-dual (up to semi-simplification).
\end{Prop}

We have the following facts about the geometry of $\mathcal{E}_{a}, \mathcal{X}_a$ and $\mathcal{S}_a$, which follow from the proof of Proposition 4.2 and Remark 5.3 in \cite{GT-selfdual}. 
\begin{Lem}\label{Lem-Hodge-num-vGT}
In the above construction, 
\begin{enumerate}
    \item $\mathcal{E}_{a}$ are rational elliptic surfaces.
    \item $\mathcal{X}_{a}$ are $K3$ surfaces with Picard number at least $19$.
    \item $\mathcal{S}_{a}$ has (complex) Hodge numbers $h^{2,0}=h^{0,2}=3$, $h^{1,1}=40$. And the Picard number for it is at least $37$.
\end{enumerate}
\end{Lem}

\subsection{Trace of \texorpdfstring{$V_{\ell}(1)$}{}}\label{section-Trace-Vl}
Now we want to deduce a trace formula of $V_{\ell}(1)$ for future use. Note that the ramified primes of $V_{\ell}$ are the divisors of $2(1+a)(1-a)$ (for notation, see Section~\ref{Section-notation}). For a fixed prime integer $p$ which is not dividing $2(1+a)(1-a)$, and let $\mathfrak{p}$ be a prime ideal in a number field $K$ and $\mathfrak{p}$ be lying above $p$. Let $q=Nm_{K/\Q}(\mathfrak{p})=p^r$ and $\rho_{\ell}:G_{\Q}\rightarrow \GL_{3}(\overline{\Q}_{\ell})$ be the semi-simplification of $V_{\ell}(1)$. Moreover, we take the notation $\overline{\mathcal{S}}_{a,t}$ (resp. $\overline{\mathcal{X}}_{a,t}$, $\overline{\mathcal{E}}_{a,t}$) to represent the reduction of ${\mathcal{S}}_{a,t}$ (resp. ${\mathcal{X}}_{a,t}$, ${\mathcal{E}}_{a,t}$) with respect to $p$.  Then the formula to compute the trace of corresponding geometric Frobenius $\Frob_{\mathfrak{p}}$ attached to $V_{\ell}$ is (see \cite[Theorem~3.5]{GT-selfdual} for more details)

\begin{equation}\label{eqn-trace-formula}
    \tr\rho_{\ell}\varepsilon_{\ell}^{-1}(\Frob_{\mathfrak{p}})=\frac{\#\overline{\mathcal{S}}_a(\F_q)-\#\overline{\mathcal{X}}_a(\F_q)}{2}.
\end{equation}
With this formula, we can compute the trace of $\Frob_{\mathfrak{p}}$ modulo an integer $m$. In order to do this, we compute $\#\overline{\mathcal{S}}_a(\F_q)-\#\overline{\mathcal{X}}_a(\F_q)$ fiberwisely with respect to $t\in \mathbb{P}^1_{\F_q}$. First, the following lemma about the fibers of $\mathcal{E}_a$ over algebraic closed field $\overline{\F}_q$ will be useful. 

\begin{Lem}\label{lem-goe-fiber}
With the above notations, assume $p\nmid 2(1+a)(1-a)$.  Let the parameter go over $ \mathbb{P}^1_{\overline{\F}_q}$, then fiber $\overline{\mathcal{E}}_{a,t}$ is smooth if and only if $t\neq 0, \pm \sqrt{-1},\pm \sqrt{\frac{1+a}{1-a}}$. 
\end{Lem}
\begin{pf}
This follows directly from that the discriminant of the fibration $\overline{\mathcal{E}}_{a,t}$ is $$\frac{64(a+1)(t^2+1)((a-1)t^2+(a+1))}{t^4} .$$
\end{pf}

In the following, we will say a fiber $\overline{\mathcal{E}}_{a,t}$ is \emph{special} if $t\in \{ 0, \pm \sqrt{-1},\pm \sqrt{\frac{1+a}{1-a}}, \infty\}$, otherwise, we say the corresponding fiber is \emph{general}. Note in particular that  $\overline{\mathcal{E}}_{a,\infty}$ is special even it is smooth. Recall the construction at the beginning of Section~\ref{section-surface-vGT}. We have degree $2$ morphisms $j: \mathbb{P}^1_z\to \mathbb{P}_u^1$ and $h: \mathbb{P}^1_u\to \mathbb{P}^1_t$. We will still keep using the same notations to refer their reduction to finite fields $\F_q$. 

Due to the way we compute $\tr\rho_{\ell}\varepsilon_{\ell}^{-1}(\Frob_{\mathfrak{p}})$, one sees that for a fiber $\overline{\mathcal{E}}_{a,t}$ only when $\sqrt{t^2+1}\in \F_q$, each component its pulling back via $h$ is defined over $\F_q$ and the right hand side of \eqref{eqn-trace-formula} has contribution to the trace. In this case, we say that $\overline{\mathcal{E}}_{a,t}$ has contribution to \eqref{eqn-trace-formula}. In the followings, we will discuss the contribution of general and special fibers respectively. Since we are only interested in the nontrivial contribution from the fibers, we will assume $\sqrt{t^2+1}\in \F_q$. 

For a general fiber $\overline{\mathcal{E}}_{a,t}$, under our assumption above, we know $h^{-1}(t)\subset \F_q$. Now the two fibers $\overline{\mathcal{X}}_{a,u}$ with $u\in h^{-1}(t)$ are defined over $\F_{q}$, and can contribute to the trace formula \eqref{eqn-trace-formula}. In addition, every $\overline{\mathcal{X}}_{a,u}$ is isomorphic to $\overline{\mathcal{E}}_{a,t}$, hence their contribution is  $2\#\overline{\mathcal{E}}_{a,t}(\F_q)/2=  \#\overline{\mathcal{E}}_{a,t}(\F_q)$ in \eqref{eqn-trace-formula}.  Similarly, when $(j\circ h)^{-1}(t)\subset \F_q$, the four fibers $\overline{\mathcal{S}}_{a,z}$ with $z\in (j\circ h)^{-1}(t)$ contribute $4\#\overline{\mathcal{E}}_{a,t}(\F_q)/2=  2\#\overline{\mathcal{E}}_{a,t}(\F_q)$ to \eqref{eqn-trace-formula}. Moreover, due to the defining equation (\ref{eqn-def-eqn-E}) we have $2|\#\overline{\mathcal{E}}_{a,t}(\F_{q})$, and due to the symmetry $\overline{\mathcal{E}}_{t}\simeq \overline{\mathcal{E}}_{a,-t}$ we have  $\#\overline{\mathcal{E}}_{a,t}(\F_{q})=\#\overline{\mathcal{E}}_{a,-t}(\F_{q})$. Hence for general $t\in \F_q$ we have 
$$
4|\#\overline{\mathcal{E}}_{a,t}(\F_{q})+\#\overline{\mathcal{E}}_{a,-t}(\F_{q}).
$$
In particular, when $\sqrt{2(1+a)}\in \F_q$, we have a $4$-torsion point $\left(1, \frac{\sqrt{2(1+a)(t^2+1)}}{t}\right)$ in each of $\overline{\mathcal{E}}_{a,\pm t}(\F_q)$, thus in this case, 
$
8|\#\overline{\mathcal{E}}_{a,t}(\F_{q})+\#\overline{\mathcal{E}}_{a,-t}(\F_{q}).
$

Now we discuss the special fibers.  
    \begin{enumerate}
        \item When $t=0$, the corresponding fiber is 
        $$
        \overline{\mathcal{E}}_{a,0}: Y^2=X^2(X+2(a+1)).
        $$
        Hence one can see that the singular point is $(0,0)$ with two tangent lines $Y=\pm \sqrt{2(a+1)}X$. Moreover, the $u$-fibers above $t=0$ are $u=\pm 2$, and the $z$-fibers are $z=\pm1 \pm \sqrt{2}$. Hence we can tabular the contribution of $\overline{\mathcal{E}}_{a,0}$ in Table~\ref{Tab-0-fiber}. In each of the first two columns of this table, we write $1$ to indicate the element at the top of this column is in $\F_q$, and $-1$ otherwise. If we leave the block unfilled, it means that the contribution is independent with the value. At the last column, we list the contribution of each situation.

        \begin{table}[h!]
            \centering
            \begin{tabular}{ |c|c | c|}
            \hline
            $\sqrt{2(1+a)}$  & $\sqrt{2}$    & contribution \\ \hline 
            $1$ & $1$  & $q$ \\ \hline
            $1$ & $-1$ & $-q$ \\ \hline 
            $-1$ & $1$ & $q+2$ \\ \hline 
            $-1$ & $-1$ & $-q+2$ \\ \hline
            \end{tabular}
            \caption{Contribution of $t=0$.}
            \label{Tab-0-fiber}
        \end{table}
        
        \item When $t=\pm i$, since the corresponding $u$-fibers are $u=\pm 2i$ and the $z$-fibers are $z=\pm i$, we know that $\overline{\mathcal{E}}_{a,\pm i}$ have no contribution.
        
        \item When $t=\pm \sqrt{\frac{1+a}{1-a}}$, the corresponding  fiber is 
        $$
        \overline{\mathcal{E}}_{a,t}: Y^2=X(X+1)^2.
        $$
        The singular point is $(-1,0)$, with two tangent lines $Y=\pm i (X+1)$. Then the $u$-fibers above are $u=\pm 2\sqrt{\frac{1+a}{1-a}}\pm 2\sqrt{\frac{2}{1-a}}$. The the $z$-fibers are $z=\pm\sqrt{\frac{1+a}{1-a}}\pm \sqrt{\frac{2}{1-a}}\pm \sqrt{\frac{4\pm 2\sqrt{2(1+a)}}{1-a}}$. Hence by the same manner as above, we have Table \ref{Tab-sp-singular-fiber}.
        
\begin{table}[h!]
\centering
\begin{tabular}{ |c|c|c| c| c| } 
 \hline
 $\sqrt{\frac{1+a}{1-a}}$ & $\sqrt{\frac{2}{1-a}}$ & $\sqrt{\frac{4\pm 2\sqrt{2(1+a)}}{1-a}}$ & $i$ & contribution  \\  \hline
 $1$ & $1$ & $1$ & $1$ & $2q$ \\ \hline
 $1$ & $1$ & $1$ & $-1$ & $2(q+2)$ \\ \hline
 $1$ & $1$ & $-1$ & $1$ & $-2q$ \\ \hline
 $1$ & $1$ & $-1$ & $-1$ & $-2(q+2)$ \\ \hline
  & $-1$ &  &  & $0$ \\ \hline
 $-1$ &  &  &  & $0$ \\ \hline
\end{tabular}
\caption{Contribution of $t=\pm \sqrt{\frac{1+a}{1-a}}$.}
\label{Tab-sp-singular-fiber}
\end{table}

\item When $t=\infty$, the corresponding fiber is smooth (recall $a\neq \pm 1$). In particular, we have 
$$
\overline{\mathcal{E}}_{a,\infty}: Y^2=X(X^2+2aX+1). 
$$
When $\sqrt{2(1+a)}\in \F_q$, we have a $4$-torsion point $(1, \sqrt{2(1+a)})$. Moreover, when $\sqrt{a^2-1}\in \F_q$, then we have three distinct $2$-torsion points over $\F_q$. The $u$-fibers above are $u=0, \infty$, and the $z$-fibers are $z=0,\infty, \pm 1$. Hence 
\begin{enumerate}
    \item When $\sqrt{2(1+a)}\notin \F_q$, $2| \#\overline{\mathcal{E}}_{a,\infty}(\F_q)$ and  $\overline{\mathcal{E}}_{a,\infty}$ contributes $\#\overline{\mathcal{E}}_{a,\infty}(\F_q)$.
    
    \item When $\sqrt{2(1+a)}\in \F_q$, $4| \#\overline{\mathcal{E}}_{a,\infty}(\F_q)$ and  $\overline{\mathcal{E}}_{a,\infty}$ contributes $\#\overline{\mathcal{E}}_{a,\infty}(\F_q)$.
    
    \item When $\sqrt{a^2-1}\in \F_q$, $4| \#\overline{\mathcal{E}}_{a,\infty}(\F_q)$ and  $\overline{\mathcal{E}}_{a,\infty}$ contributes $\#\overline{\mathcal{E}}_{a,\infty}(\F_q)$.
    
    \item When $\sqrt{2(1+a)}\in \F_q$ and $\sqrt{a^2-1}\in \F_q$, $8| \#\overline{\mathcal{E}}_{a,\infty}(\F_q)$ and  $\overline{\mathcal{E}}_{a,\infty}$ contributes $\#\overline{\mathcal{E}}_{a,\infty}(\F_q)$.
    
\end{enumerate}
       
    \end{enumerate}
    
\begin{Prop}\label{Prop-tr-mod8-p^2}
    Let $p\nmid 2\ell(1+a)(1-a)$ and let $\left(\frac{2(1+a)}{p}\right)=\left(\frac{2(1-a)}{p}\right)=-1$. Take $q=p^2$, then 
    $$
    \tr\rho_{\ell}\varepsilon_{\ell}^{-1}(\Frob_{\mathfrak{p}})=-q\ (\bmod\ 8).
    $$
\end{Prop}

\begin{pf}
According to the above discussion, we can see that when $q=p^2$, then $$i, \sqrt{2}, \sqrt{1+a}, \sqrt{1-a}\in \F_q.$$ 
This means 
\begin{enumerate}
    \item The contribution of general fiber is $0$ $(\bmod\ 8)$. 
    \item $t=0$ contributes $q$ to trace. 
    \item If $\sqrt{\frac{4\pm 2\sqrt{2(1+a)}}{1-a}}\in \F_q$, then $t=\pm \sqrt{\frac{1+a}{1-a}}$ contribute $2q$ for trace, otherwise $-2q$. 
    \item The contribution of $t=\infty$ is $0$ $(\bmod\ 8)$. 
\end{enumerate}
To determine whether $\sqrt{\frac{4\pm 2\sqrt{2(1+a)}}{1-a}}\in \F_q$, we need to tell whether $\F_p(\sqrt{{4\pm 2\sqrt{2(1+a)}}})\subset \F_q$. Notice that $\alpha:=\sqrt{4+ 2\sqrt{2(1+a)}}$ is a root of the polynomial $T^4-8T^2+8(1-a)$. In fact, let $\beta=\sqrt{4- 2\sqrt{2(1+a)}}$, then the four roots of this polynomial are $\pm \alpha$ and $\pm \beta$. Now let $\sigma\in \Gal(\F_p(\alpha)/\F_p)$ be a generator, then if $\sigma$ has order $2$, then it either exchanges $\alpha$ with $-\alpha$, or exchanges $\alpha$ with one of $\pm \beta$. For the former, we know that it means $\alpha^2\in \F_p$, i.e. $\sqrt{2(1+a)}\in \F_p$. For the later, we know it means $\alpha\beta\in \F_p$, i.e. $\sqrt{2(1-a)}\in \F_p$. If $\sigma$ has order $1$, i.e. $\alpha\in \F_p$, then $\alpha^2=4-2\sqrt{2(1+a)}\in \F_p$, and thus $\sqrt{2(1+a)}\in \F_p$.  Hence under the assumption of the proposition we know that $\sqrt{\frac{4\pm 2\sqrt{2(1+a)}}{1-a}}\in \F_q$ is not in $\F_q$, and thus
$$
\tr\rho_{\ell} \varepsilon_{\ell}^{-1}(\Frob_{\mathfrak{p}})=q-2q =-q\ \mod 8. 
$$
\end{pf}

\subsection{Proof of Theorem \ref{Thm-con-Tate-conj-}}\label{section-pf-irr-vGT} 
Now we want to prove Theorem \ref{Thm-con-Tate-conj-}, i.e., $\mathcal{S}_{a}$ satisfies the conditions $(\ast)$ and $(\ast\ast)$ of Corollary \ref{Cor-Tate-conj-general}. Recall that 
\[
    H^2_{\rm{\acute{e}t}}((\mathcal{S}_{a})_{\overline{\Q}}, \Q_{\ell}(1))\cong (\NS((\mathcal{S}_{a})_{\overline{\Q}})\otimes \Q_{\ell})\oplus \Tr_{\ell}(\mathcal{S}_{a}).
\]
By the construction of $V_{\ell}$ and $\overline{V}_{\ell}$ in Section 4.2 and Lemma~\ref{Lem-Hodge-num-vGT} (2), we have 
\[
   \Tr_{\ell}(\mathcal{S}_{a})^{\rm{ss}}\subseteq  V_{\ell}(1)^{\rm{ss}}\oplus \overline{V}_{\ell}(1)^{\rm{ss}}\oplus U_{\ell}(1)^{\rm{ss}},
\]
as $\ell$-adic representations of $G_{\Q}$,
where $U_{\ell}(1)^{\rm{ss}}$ is the transcendental part of the $H^2_{\rm{\acute{e}t}}((\mathcal{X}_{a})_{\overline{\Q}}, \Q_{\ell}(1))$. Then $\{U_{\ell}(1)^{\rm{ss}}\}_{\ell}$ is a rank $2$ or $3$ weakly compatible system of $\ell$-adic representations of $G_{\Q}$ defined over $\Q$ since the $K3$ surface $\mathcal{X}_{a}$ has Picard number $19$ or $20$. Moreover, due to the fact that the Tate conjecture is known for $K3$ surface, $U_{\ell}(1)^{\rm{ss}}$ is absolutely irreducible as $G_{\Q}$-representation. 

\begin{Prop}\label{Prop-star}
    For each $a\in \Q\setminus \{\pm 1\}$, the surface $\mathcal{S}_{a}$ satisfies the condition $(\ast)$ of Corollary \ref{Cor-Tate-conj-general}.
\end{Prop}
\begin{proof}
By the above discussion, $\{U_{\ell}(1)\}_{\ell}$ is a rank $2$ or $3$ regular weakly compatible system of self-dual $\ell$-adic representations of $G_{\Q}$ defined over $\Q$. 

Considering the representations $V_{\ell}$ and $\overline{V}_{\ell}$. They are motivically defined, and the complex Hodge number are $h^{2,0}=h^{1,1}=h^{0,2}=1$  (see \cite[proof of Proposition~4.2]{GT-selfdual}). So $\{V_{\ell}(1)^{\rm{ss}}\}_{\ell}$, and $\{\overline{V}_{\ell}(1)^{\rm{ss}}\}_{\ell}$ are also rank $3$ self-dual regular weakly compatible system of $\ell$-adic representations of $G_{\Q}$ defined over $\Q$.
\end{proof}

Before we consider the condition $(\ast\ast)$ of Corollary \ref{Cor-Tate-conj-general}, we need to state a lemma for later use. 
    
    \begin{Lem}\label{Lem-same-det}
        Let $\rho: G_{\Q}\rightarrow\GL_{3}(\overline{\Q}_{\ell})$ be an self-dual $\ell$-adic representation and $\rho\cong \psi\oplus r$ decomposes into the direct sum of two irreducible $\overline{\Q}_{\ell}$-subrepresentations with $\dim \psi=1$ and $\dim r=2$. Then, for an element $g$, $\det r(g)=1$ in any one of the following cases:
        \begin{enumerate}
            \item[(a)] $\tr \rho(g^2)\neq 3 \ (\bmod\ m)$ for some integer $m\geq 5$. 
            \item[(b)] $\tr\rho(g)\neq \pm 1$. 
        \end{enumerate}
    \end{Lem}
    \begin{proof}
    Since $\rho$ is self-dual, $
    \psi\oplus r\cong \rho\cong \rho^{\ast}\cong \psi^{\ast}\oplus r^{\ast}$. By Jordan-Holder Theorem, $\psi\cong \psi^{\ast}$ and $r\cong r^{\ast}$, i.e., $\psi$ and $r$ are self-dual. In particular, $\psi$ and $\det r$ are quadratic characters.
    
    Also notice that, since $\rho$ is self-dual, its image is in the orthogonal group $O_{3}(\overline{\Q}_{\ell})$. Then, for any $g\in G_{\Q}$, $\rho (g)$ is a diagonalizable matrix. Assume that the Jordan canonical form of $\rho (g)$ is $\left(\begin{smallmatrix}
    \alpha & 0 & 0 \\
    0 & \beta & 0 \\
    0 & 0 & \gamma
    \end{smallmatrix}\right)$. Since $\rho(g)$ is similar to $(\rho(g)^{-1})^{t}$. Then there are two cases.
\begin{enumerate}
    \item $\alpha=\pm 1$, $\beta =\pm 1$, and $\gamma =\pm 1$.
    \item $\alpha\beta =1 (\alpha\not=\pm 1)$, and $\gamma =\pm 1 $.
\end{enumerate} 

Observe that, in case $(2)$, $\det r(g)=\alpha \beta =1$ since $\psi$ is a quadratic character. And, in case (1), if $\alpha=\beta=\gamma=\pm 1$, $\det r(g)=1$. Then this lemma follows.
    \end{proof}

\begin{Prop}\label{Prop-stars}
    For each $a\in \Q$, if $a\equiv 2,3 \mod 5$, and none of $2(1+a)$ or $2(1-a)$ is a square in $\Q$, the surface $\mathcal{S}_{a}$ satisfies the condition $(\ast\ast)$ of Corollary \ref{Cor-Tate-conj-general}.
\end{Prop}

We denote $V_{\ell}(1)$ by $\rho_{\ell}$ for any prime $\ell$. By Proposition 4.2, $\rho_{\ell}^{\rm{ss}}$ is self-dual, If $\rho_{\ell}^{\rm{ss}}$ is decomposed into irreducible $\overline{\Q}_{\ell}$-subrepresentations as follows
\[
\rho_{\ell}^{\rm{ss}}\cong \psi_{\ell}\oplus r_{\ell}
\]
with $\dim\psi_{\ell}=1$ and $\dim r_{\ell}=2$. We want to prove that $\det r_{\ell}=1$. 
\begin{proof}
By the proof of Lemma~\ref{Lem-same-det}, we know that $\det r_{\ell}$ is a quadratic character. Thus there is an integer $D$ such that
\[
\det r_{\ell}(\Frob_{p}) =\left(\frac{D}{p}\right)
\]
for prime $p\nmid D$.

We first prove that, if neither $2(1+a)$ nor $2(1-a)$ is a square in $\Q$, then $D$ is 1 or  $1-a^2$ (up to a square). Suppose that this is not the case, then by Chinese reminder theorem we can find a prime integer $p$ such that 
    $$
    \left(\frac{D}{p}\right)=-1, \quad \left(\frac{2(1+a)}{p}\right)=\left(\frac{2(1-a)}{p}\right)=-1.
    $$
    So $\det r_{\ell}(\Frob_{p})=\left(\frac{D}{p}\right)=-1$. On the other hand, by Proposition \ref{Prop-tr-mod8-p^2}, $\left(\frac{2(1+a)}{p}\right)=\left(\frac{2(1-a)}{p}\right)=-1$ implies that $\tr \rho_{\ell}(\Frob_p^2)=-1\ (\bmod\ 8)$. Then according to Lemma \ref{Lem-same-det}, we have $\det r_{\ell}(\Frob_p)=1$. This is a contradiction. So $D$ is $1$ or $1-a^2$ (up to a square). 
    
    Secondly, we want to show that, if $a\equiv 2,3\  (\bmod 5)$ and neither $2(1+a)$ nor $2(1-a)$ is a square in $\Q$, $D$ is $1$ up to a square. Suppose not, then $\det r$ has to be $\left(\frac{1-a^2}{\bullet}\right)$. Now let $p=5$, then $1-a^2=2\in \F_p$, hence $\det r(\Frob_{5})=\left(\frac{1-a^2}{p}\right)=-1$. Then by part (2) of Lemma \ref{Lem-same-det}, we know that $\tr \rho_{\ell}(\Frob_5)$ has to be $\pm 1$, or equivalently, $\tr\rho_{\ell}\varepsilon_{\ell}^{-1}(\Frob_5)=\pm 5$. In the following, we will use the trace formula \eqref{eqn-trace-formula} to show that this is impossible by calculation. 
    
    Notice that over $\F_5$, the general fiber (cf. Lemma \ref{lem-goe-fiber}) are $t=1,4$. None of them contribute to the trace $\tr \rho_{\ell}\varepsilon_{\ell}^{-1}(\Frob_5)$ since $t^2+1=2\in \F_5$ is not a square. So we only need to consider the singular fiber $t=0$ and $\infty$. Suppose $a\equiv2\ (\bmod\ 5)$, then by Table \ref{Tab-0-fiber} we know that the fiber $t=0$ contributes $-5$ to $\tr\rho_{\ell}\varepsilon_{\ell}^{-1}(\Frob_5)$. By point counting, we know that the fiber $t=\infty$ contributes $8$ to $\tr\rho_{\ell}\varepsilon_{\ell}^{-1}(\Frob_5)$. Hence now we have 
    $$
    \tr\rho_{\ell}\varepsilon_{\ell}^{-1}(\Frob_5)=8-5=3\neq \pm 5. 
    $$ Similarly, suppose $a=3$, then the fiber $t=0$ contributes $-7$, and $t=\infty$ contributes $8$ to $\tr\rho_{\ell}(\Frob_5)$. Hence 
    $$
    \tr\rho_{\ell}\varepsilon_{\ell}^{-1}(\Frob_5)=8-7=1\neq \pm 5. 
    $$
    By all above, one sees that for each $a=2$ or $3\ (\bmod\ 5)$, we have $\tr\rho_{\ell}\varepsilon_{\ell}^{-1}(\Frob_p)\neq \pm 5$, hence we obtain a contradiction. So we are done. 
\end{proof}

\begin{Rmk}\label{Rmk-phi-trivial-mod-7} 
In fact, $\left(\frac{a}{5}\right)=-1$ in Theorem~\ref{Thm-con-Tate-conj-} is only a technical condition and seems easy to generalize. For instance one can also show that  if $a=3$ or $4$ $(\bmod\ 7)$, then $\det r_{\ell}$ is trivial. 
\end{Rmk}

Now combine all the results above, we are able to give a proof to Theorem \ref{Thm-con-Tate-conj-}.

\begin{pf}[Proof of Theorem~\ref{Thm-con-Tate-conj-}]
By Propositions \ref{Prop-star} and \ref{Prop-stars}, the surface $\mathcal{S}_{a}$ satisfies all the conditions of Corollary \ref{Cor-Tate-conj-general}. Then by Corollary~\ref{Cor-Tate-conj-general}, for a Dirichlet density one subset of primes $\ell$, the corresponding Tate conjecture for $\mathcal{S}_{a}$ is true.
\end{pf}

\section*{Acknowledgements}
Our project is inspired by the work of Bert van Geemen and Jaap Top on construction of 3-dimensional Galois representations. We thank Bert van Geemen and Frank Calegari for valuable feedback. The first author would like to thank Siman Wong for introducing the question about irreducibility of Galois representations, which leads us to this project. The second author would like to thank Joel Specter and Dingxin Zhang for helpful discussions on the Galois representations and pure motives. We would also like to thank Jeff Achter, Yuan Ren, David Savitt, Shaoyun Yi, for helpful conversations on this project. We are grateful to the anonymous referees for their helpful comments.

\bibliographystyle{amsalpha}
\bibliography{mybib}

\end{document}